\documentclass[11pt]{article}
\usepackage[active]{srcltx}
\usepackage[usenames]{color}

\usepackage{amstext,amsfonts,amsthm,graphicx,amssymb,amscd,epsfig}
\usepackage{amsmath}
\usepackage{mathrsfs}
\usepackage[ansinew]{inputenc}    %

\newtheorem{prop}{Proposition}[section]
\newtheorem{theo}[prop]{Theorem}
\newtheorem{cor}[prop]{Corollary}

\newtheorem{rem}[prop]{Remark}

\numberwithin{equation}{section}
\newcommand{\R}{\mathbb R}

\date{}

\begin{document}
\title{On the geometry of folded cuspidal edges}
\author{Ra\'{u}l Oset Sinha\footnote{Partially supported by DGICYT Grant MTM2015--64013--P.}\, and\, Kentaro Saji
\footnote{Supported by JSPS KAKENHI Grant Number JP26400087.}}

\maketitle
\begin{abstract}
We study the geometry of cuspidal $S_k$ singularities in
$\mathbb R^3$ obtained by folding generically a cuspidal edge. In
particular we study the geometry of the cuspidal cross-cap $M$, i.e.
the cuspidal $S_0$ singularity. We study geometrical invariants
associated to $M$ and show that they determine it up to order 5. We
then study the flat geometry (contact with planes) of a generic
cuspidal cross-cap by classifying submersions which preserve it and
relate the singularities of the resulting height functions with the
geometric invariants.
\end{abstract}

\renewcommand{\thefootnote}{\fnsymbol{footnote}}
\footnote[0]{2010 Mathematics Subject classification 57R45, 53A05.}

\section{Introduction}\label{sec:intro}
Given a parametrisation $\phi:U\subset\mathbb R^2\rightarrow \mathbb
R^3$ of a surface $N$, where $U$ is an open set, we say that
$N$ is a cuspidal $S_k$ singularity if it admits a
parametrisation $\phi$ which is $\mathcal A$-equivalent (equivalent
by diffeomorphisms in source and target) to
$f(x,y)=(x,y^2,x^{k+1}y^3+y^5)$. In the particular case of $k=0$,
the cuspidal $S_0$-singularity is the cuspidal cross-cap (or folded
umbrella), which we denote by $M$. In this case, the image of $f$
resembles that of the Whitney umbrella but it contains a cuspidal
edge transversal to the double point curve (see Figure \ref{ccc}).
Cuspidal $S_k$ singularities are types of frontal singularities and
the cuspidal cross-cap in particular naturally appears in different
contexts of differential geometry. For example, it is a generic
singularity of bicaustics, the surface drawn by the cuspidal edges
of a 1-parameter family of caustics in 3-space \cite{arnold1}. It is
also the singularity which appears in a developable surface of a
space curve at a point of zero torsion \cite{cleave}. It even
appears as a type of generic singularity in dynamical systems
related to relaxational equations (\cite{davydov}, \cite{zak}).

\begin{figure}
\begin{center}
\includegraphics[width=0.5\linewidth]{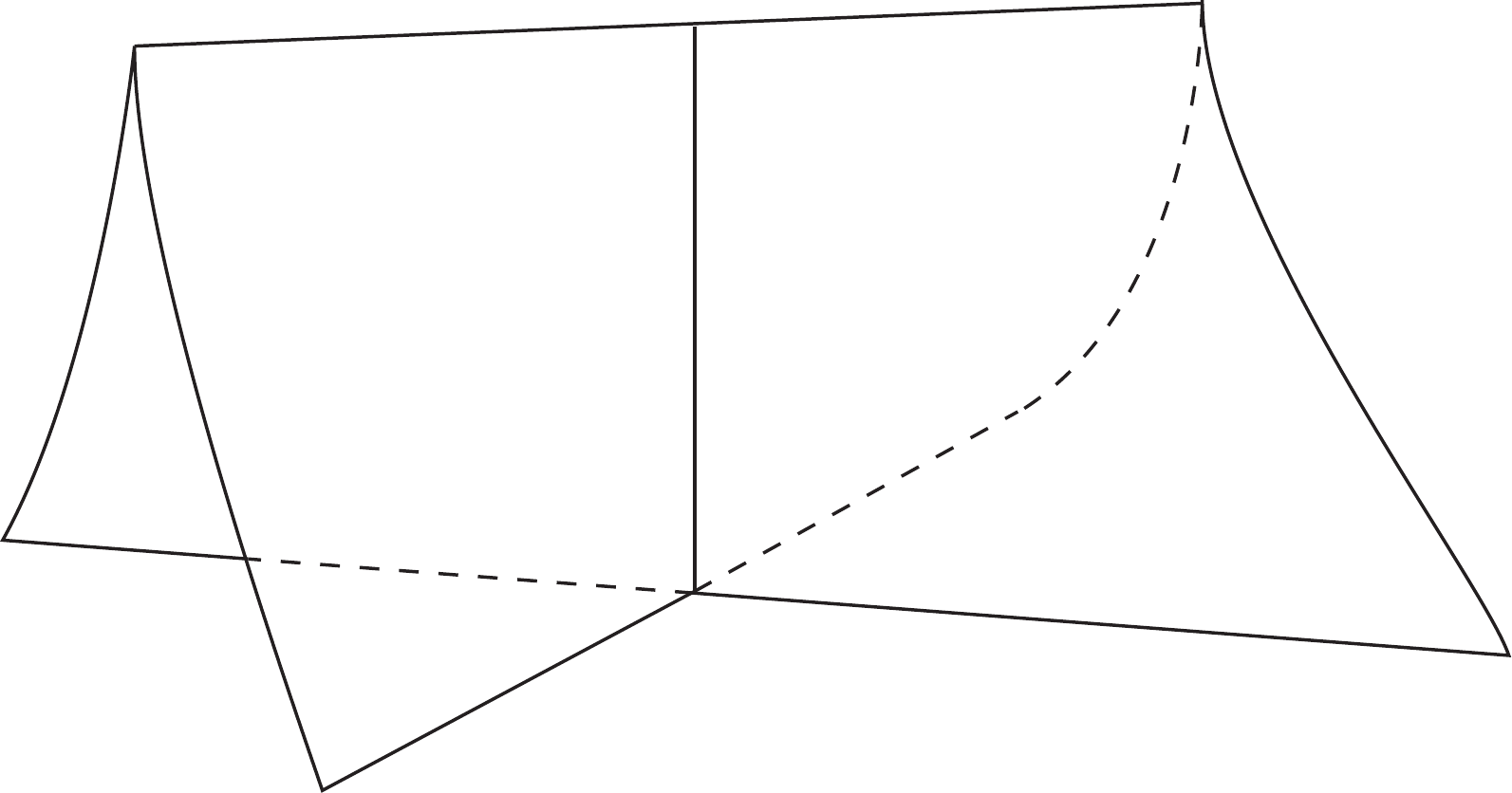}
\caption{The standard cuspidal cross-cap.}
\label{ccc}
\end{center}
\end{figure}

There has been a recent impulse in the study of the differential geometry of singular surfaces. Defining new geometric invariants
(namely, invariants of surfaces under the action of $SO(3)$ in $\R^3$)
and applying singularity theory techniques has become a crucial subject to attain this goal.
For instance, there have been many developments considering the geometry of the cross-cap (Whitney umbrella) (\cite{barajaskabata,brucewest,diastari,fukuihasegawa,hasegawahondaetal1,hasegawahondaetal2,hnuy,nunotari,oliver,osettari1}). Considering the most simple type of wave front singularity, the cuspidal edge, there have also been many significant advances (\cite{kokubuetal,martinssaji1,martinssaji2,naokawaetal,osettari2,sajietalannals,teramoto}). Many papers are devoted to the study of singularities of wave fronts or frontals in general (\cite{msuy,sajietalmathproc,sajietalgeomanal}). Also, in \cite{martinsnuno}, the geometry of corank 1 surface singularities is studied in general.

In order to study the generic geometry of a certain singular surface
it is convenient to have a normal form with arbitrary coefficients
which give the necessary freedom to capture all the generic
geometry. Such a normal form must be obtained by applying
diffeomorphisms in the source, but only isometries in the target. In
\cite{west}, a normal form was obtained for the cross-cap. In
\cite{martinssaji1}, Martins and the second author obtain a normal
form for the cuspidal edge.
In the case of the cross-cap the authors consider the normal form
modulo order 3 terms, and in the cuspidal edge case the normal form
is considered up to order 3 terms. However, for more degenerate
singularities the order one must consider in order to capture
interesting geometrical features can grow considerably. For
example, taking a generic section of the cuspidal cross-cap $M$
through the origin yields a ramphoid (or (2,5)-)cusp, therefore, it
is reasonable to think that a normal form must consider up to order
5 terms. We avoid this difficulty by taking an alternative approach:
given a generic cuspidal edge, if we take a plane transversal to the
cuspidal edge curve and fold the cuspidal edge along that plane we
obtain a cuspidal cross-cap. This was first noticed by Arnol'd in
\cite[page 120]{arnold} and generalized by the second author in
\cite{sajicuspidal}. If the order of contact of the cuspidal
edge with the folding plane is $k+1$, the resulting surface has an
$S_k$ singularity.

The idea of folding maps (more precisely, studying geometry by
considering its symmetries) dates back to Klein's Erlangen
programme. In \cite{brucewilkinson} Bruce and Wilkinson studied this
phenomenon from the singularity theory point of view and it is
starting to be of interest for geometers in singularity theory again
(\cite{barajas,izumiyaetalosaka}). More recently,
Pe\~nafort-Sanchis has studied (generalized) reflection maps as a source of corank 2 maps and has proved
L\^e's conjecture about the injectivity of corank 2 maps for this
class of maps (\cite{penafort}). In our case, considering the
cuspidal cross-cap obtained by folding a cuspidal edge is obviously
more restrictive than studying a generic cuspidal cross-cap. We
shall show that from the point of view of flat geometry, there is only one generic singularity of the height function which is not
captured by this process. Namely, the reflecting plane is the
tangent cone of the resulting cuspidal cross-cap and all of the
surface is on one side of this plane, but generically the surface
could be on both sides of the tangent cone, as we shall show.
This suggests that the set of germs of surfaces with a cuspidal cross-cap obtained by folding a cuspidal edge has (in a certain sense) codimension 1 in the set of germs of surfaces with a cuspidal cross-cap singularity.
However, there are more advantages than disadvantages (besides the
fact of not needing a special normal form) since we can relate the
geometry of the cuspidal cross-cap with that of the cuspidal edge it
comes from. In fact, we find a relation between a generic
singularity of a height function on a generic cuspidal cross-cap and
the torsion of a certain curve in the cuspidal edge before folding
it.

The paper is organized as follows: Section 2 explains the setting and studies geometric invariants of the cuspidal cross-cap, and cuspidal $S_k$ singularities obtained by folding a cuspidal edge. We consider geometric invariants of the cuspidal edge in the cuspidal cross-cap, the double point curve and of the ramphoid cusp obtained by a generic section through the origin. Some relations are given amongst these invariants and it is shown which of these invariants determine the cuspidal cross-cap up to order 5. Section 3 is devoted to the classification of submersions preserving $M$. The singularities of these submersions model the singularities of the height functions on $M$ and capture the geometry of the contact of $M$ with planes. These singularities are related to the geometric invariants considered in Section 2. We then study the duals of the different generic $M$. Finally, in Section 4, we consider the geometry of the tangent developable of a space curve at a point of zero torsion.

\emph{Acknowledgements:} The authors would like to thank Farid Tari
for helpful discussions and the referees for valuable suggestions which improved the scope and presentation of the results.
\section{Geometrical invariants}

\subsection{Normal form of the cuspidal edge and its invariants}
A map-germ
$f:(\mathbb R^2,0)\rightarrow (\mathbb R^3,0)$
is a {\it frontal\/} if
there exists a well defined normal unit vector field $\nu$ along $f$,
namely, $|\nu|=1$ and for any $X\in T_p\R^2$, $df_p(X)\cdot\nu(p)=0$.
A frontal $f$ with a unit normal unit vector field $\nu$
is a {\it front} if the pair $(f,\nu)$ is an immersion.
Since at a cuspidal edge $f:(\mathbb R^2,0)\rightarrow (\mathbb R^3,0)$,
there is always a well defined normal unit vector field $\nu$ along $f$,
and the pair $(f,\nu)$ is an immersion, a cuspidal edge is a front.
On the other hand,
at a cuspidal cross-cap $f:(\mathbb R^2,0)\rightarrow (\mathbb R^3,0)$,
there is always a well defined normal unit vector field $\nu$ along $f$,
but the pair $(f,\nu)$ is not an immersion,
a cuspidal cross-cap is a frontal but not a front.
Let $f:(\mathbb R^2,0)\rightarrow (\mathbb R^3,0)$ be a frontal
with a normal unit vector field $\nu$.
Consider the function $\lambda=\det(f_x,f_y,\nu)$, where $(x,y)$ are the coordinates of $\mathbb R^2$.
Then $S(f)=\{\lambda^{-1}(0)\}$, where $S(f)$ is the set of singular
point of $f$.
A singular point $q$ is non-degenerate if $d\lambda(q)\neq 0$.
If $q$ is a non-degenerate singular point,
there is a well defined vector field $\eta$ in $\mathbb R^2$,
such that $df(\eta)=0$ on $S(f)$.
Such a vector field is called a {\it null vector field}.
A singular point $q$ is called of {\it first kind\/}
if $\eta\lambda(q)\ne0$.
A singular point $q$ is of first kind of a front if $f$
is a cuspidal edge (\cite{kokubuetal}).
Let $f:(\mathbb R^2,0)\rightarrow (\mathbb R^3,0)$ be a frontal
with a normal unit vector field $\nu$, and
$0$ a singular point of the first kind.
Since $\eta$ is transversal to $S(f)$, we can consider another vector field $\xi$ which is tangent to $S(f)$ and such that $(\xi,\eta)$ is positively oriented. Such a pair of vector fields is called an adapted pair. An adapted coordinate system $(u,v)$ of $\mathbb R^2$ is a coordinate system such that $S(f)$ is the $u$-axis, $\partial_v$ is the null vector field and there are no singular points besides the $u$-axis. Let $\gamma$ be a parametrisation of the singular curve $S(f)$ and let $\widehat\gamma=f\circ\gamma$.

In \cite{martinssaji1} certain geometric invariants of cuspidal edges are studied. Amongst them are the singular curvature, the limiting normal curvature, the cuspidal curvature and the cusp-directional torsion ($\kappa_s$, $\kappa_{\nu}$, $\kappa_c$ and $\kappa_t$, resp.), and these are given as follows:
\begin{equation}\label{eq:invdef1}
\kappa_s(t)=
\displaystyle
\operatorname{sgn}(d\lambda(\eta))\frac
{\det(\widehat\gamma'(t),\widehat\gamma''(t),\nu(\gamma(t)))}
{|\widehat\gamma'(t)|^3},\quad
\kappa_{\nu}(t)=
\displaystyle
\frac{\langle \widehat\gamma''(t),\nu(\gamma(t))\rangle}
{|\widehat\gamma'(t)|^2},
\end{equation}
and
\begin{align}
\label{eq:invdef2}
\kappa_c(t)&=
\displaystyle
\frac{|\xi f|^{3/2}\det(\xi f,\eta^2 f,\eta^3 f)}
{|\xi f\times\eta^2 f|^{5/2}}(\gamma(t)),\displaybreak[3]\\
\label{eq:invdef3}
\kappa_t(t)&=
\displaystyle
\frac{\det(\xi f,\eta^2 f,\xi\eta^2 f)}
{|\xi f\times\eta^2 f|^2}(\gamma(t))
-
\frac{\det(\xi f,\eta^2 f,\xi\eta^2 f)\langle \xi f,\eta^2 f \rangle}
{|\xi f|^2|\xi f\times\eta^2 f|^2}(\gamma(t)),
\end{align}
where $'$ stands for the differential with respect to the considered variable,
and $\zeta^i f$ stands for the $i$ times directional derivative
of $f$ by the vector field $\zeta$.
A detailed description and geometrical interpretation of $\kappa_s$ and $\kappa_{\nu}$ can be found in \cite{sajietalannals}, of $\kappa_c$ in \cite{msuy} and of $\kappa_t$ in \cite{martinssaji1}.
Since $\xi f\times\eta^2 f\ne0$ for singularities of the first kind,
these invariants can be also defined for the singularities of the first kind.

A normal form for a cuspidal edge is obtained
in \cite[Theorem 3.1]{martinssaji1}.
The same proof works for the case of a singular point
of the first kind, and we obtain:
\begin{prop}
Let $f:(\mathbb R^2,0)\rightarrow (\mathbb R^3,0)$ be a frontal
with a normal unit vector field $\nu$.
Let $0$ be a singular point of the first kind.
Then there exist a coordinate system $(x,y)$ on $(\R^2,0)$
and an isometry-germ\/ $\Phi:(\R^3,0)\to(\R^3, 0)$
such that
\begin{equation}
\label{eq:normal}
\Phi\circ f(x,y)
=
\left(
x,a(x)+\dfrac{y^2}{2},
b_0(x)+b_1(x)y^2+b_2(x)y^3+b_3(x,y)y^4\right),
\end{equation}
where
$a,b_0,b_1,b_2,b_3$ be smooth functions
such that
$
a(0)=a'(0)=b_0(0)=b_0'(0)=b_1(0)=0.
$
\end{prop}
Let $f$ be a map-germ given by \eqref{eq:normal}.
If $b_2(0)\ne0$, then $f$ is a cuspidal edge,
and
if $b_2(0)=0$, $b_2'(0)\ne0$, then $f$ is a cuspidal cross-cap.
Moreover, if $b_2^{(i)}(0)=0$ $(i=1,\ldots,k)$ and
$b_2^{(k+1)}(0)\ne0$, then $f$ is $\mathcal A$-equivalent to
$(x,y^2,x^{k+1}y^3\pm y^5)$
(i.e. a {\it cuspidal\/ $S_k^\pm$ singularity},
and a cuspidal $S_0$ singularity is a cuspidal cross-cap).
Furthermore, we have:
$$
|\kappa_s(0)|=|a''(0)|,\quad
\kappa_\nu(0)=b_0''(0),\quad
\kappa_c(0)=6b_2(0),\quad
\kappa_t(0)=2b_1'(0).
$$
If $b_2(0)=0$, (namely, non-cuspidal edge), one can define other invariants.
Let $f$ be a map-germ given by \eqref{eq:normal}
with $b_2(0)=0$.
Then one can take a null vector satisfying
$
\xi f\cdot\tilde\eta^2 f(0)=
\xi f\cdot\tilde\eta^3 f(0)=0.
$
Then there exist $l$ such that
$\tilde\eta^3 f(0)=
l\tilde\eta^2 f(0)$.
Following \cite{hs},
we define two real numbers by
\begin{align}
\label{eq:bias}
B=&
\dfrac{|\xi f|^2\det\Big(\xi f,\
\tilde\eta^2 f,\ \tilde\eta^4f\Big)}
{|\xi f\times \tilde\eta^2 f|^{3}}\Bigg|_{(u,v)=0},\\
\label{eq:kcr}
\kappa_c^{{\rm r}}=&
\dfrac{|\xi f|^{5/2}\det\Big(
\xi f,\ \tilde\eta^2 f,\
3\tilde\eta^5 f-10\,l\,\tilde\eta^4 f
\Big)}
{|\xi f\times \tilde\eta^2 f|^{7/2}}\Bigg|_{(u,v)=0}.
\end{align}
$B$ and $\kappa_c^{{\rm r}}$
do not depend on the choice of $(\xi,\tilde\eta)$.
The invariant $B$ measures the bias of a curve around the singular point
and it is called {\it bias},
and
$\kappa_c^{{\rm r}}$ measures wideness of the cusp
and it is called {\it secondary cuspidal curvature}.
See \cite{hs} for details.

\subsection{Invariants of singular space curve}\label{sec:spcurv}
In order to study special curves on the
cuspidal $S_k^\pm$ singularities,
we consider geometric invariants of singular
space curves.
Let $\gamma:(\mathbb R,0)\to(\mathbb R^3,0)$ be a curve and assume that
$\gamma'(0)=(0,0,0)$. The point $0$ is called an {\it $A$-type}\/ point if
$\gamma''(0)\ne(0,0,0)$, and $0$ is called $(2,3)$-type if
$\gamma''(0)\times\gamma'''(0)\ne(0,0,0)$.

Let $0$ be an $A$-type singular
point of $\gamma$, then, following \cite{martinssaji2},
we define
$$
\kappa_{sing}(\gamma)
=
\dfrac{|\gamma''(0)\times \gamma'''(0)|}
{|\gamma''(0)|^{5/2}}.
$$
Moreover, let $0$ be a $(2,3)$-type singular point of $\gamma$,
then we define
$$
\tau_{sing}(\gamma)
=
\dfrac{\sqrt{|\gamma''(0)|}\det(\gamma''(0),\gamma'''(0),\gamma^{(4)}(0))}
{|\gamma''(0)\times \gamma'''(0)|^2},
$$
where $f^{(i)}$ stands for the $i$-th derivative of $f$
with respect to the considered variable.
We set
$$
\sigma_{sing}(\gamma)=
\dfrac{
\left(
\langle\gamma''(0)\times\gamma'''(0),\gamma''(0)\times\gamma^{(4)}(0)\rangle
-2\dfrac{|\gamma''(0)\times\gamma'''(0)|^2\langle\gamma''(0),\gamma'''(0)\rangle}
{\langle\gamma''(0),\gamma''(0)\rangle}\right)}
{\langle\gamma''(0),\gamma''(0)\rangle^{11/4}}.
$$
If two curves
$\gamma_1,\gamma_2:(\mathbb R,0)\to(\mathbb R^3,0)$
satisfy
$\kappa_{sing}(\gamma_1)=\kappa_{sing}(\gamma_2)$,
$\tau_{sing}(\gamma_1)=\tau_{sing}(\gamma_2)$ and
$\sigma_{sing}(\gamma_1)=\sigma_{sing}(\gamma_2)$,
then there exists an isometry and parameters $t_1$, $t_2$
such that
$j^3\gamma_1(0)=j^3\gamma_2(0)$
with respect to the parameters
$t_1$, $t_2$.
Thus the invariants
$\{\kappa_{sing}(\gamma),\sigma_{sing}(\gamma),\tau_{sing}(\gamma)\}$
can be used as invariants
for $(2,3)$-type singular
space curves up to fourth degree.
See \cite{martinssaji2} for details.

\subsection{Folded cuspidal edge}
Let $f:(\R^2,0)\to(\R^3,0)$ be the map-germ defined
by the right-hand side of \eqref{eq:normal}.
Let $(X,Y,Z)$ denote the coordinate system
on $(\R^3,0)$.
We now consider the fold map whose singular set is
the reflecting hyperplane $Z=0$, given by
$$
g(X,Y,Z)=(X,Y,Z^2)
$$
and consider the cuspidal edge folded along the plane $Z=0$,
i.e., the composition
\begin{equation}\label{eq:folding}
\phi=g\circ f
\end{equation}
of the parametrisation of the
cuspidal edge with the fold map.
We may assume that
$a(0)=a'(0)=b_0(0)=0,\ b_2(0)\ne0$
without loss of generality.
In order for $\phi=g\circ f$ to be a cuspidal $S_k$ singularity,
we need the condition
$$
b_0(0)=\cdots=b_0^{(k)}(0)=0,b_0^{(k+1)}(0)\ne0,
$$
(see \cite[Theorem 3.2]{sajicuspidal}) which is equivalent to the
cuspidal edge curve having a $(k+1)$th degree contact with the
plane $Z=0$ (see Figure \ref{folded}). By Shafarevich \cite{Sha}, if
${\mathcal X}$ is an irreducible affine variety in $\mathbb R^n$
defined by the ideal $I$ then the equations for the tangent cone of
${\mathcal X}$ are the lowest degree terms of the polynomials in
$I$. Therefore, the tangent cone of the cuspidal $S_k$ singularity
is the plane $Z=0$.

\begin{figure}
\begin{center}
\includegraphics[width=0.4\linewidth]{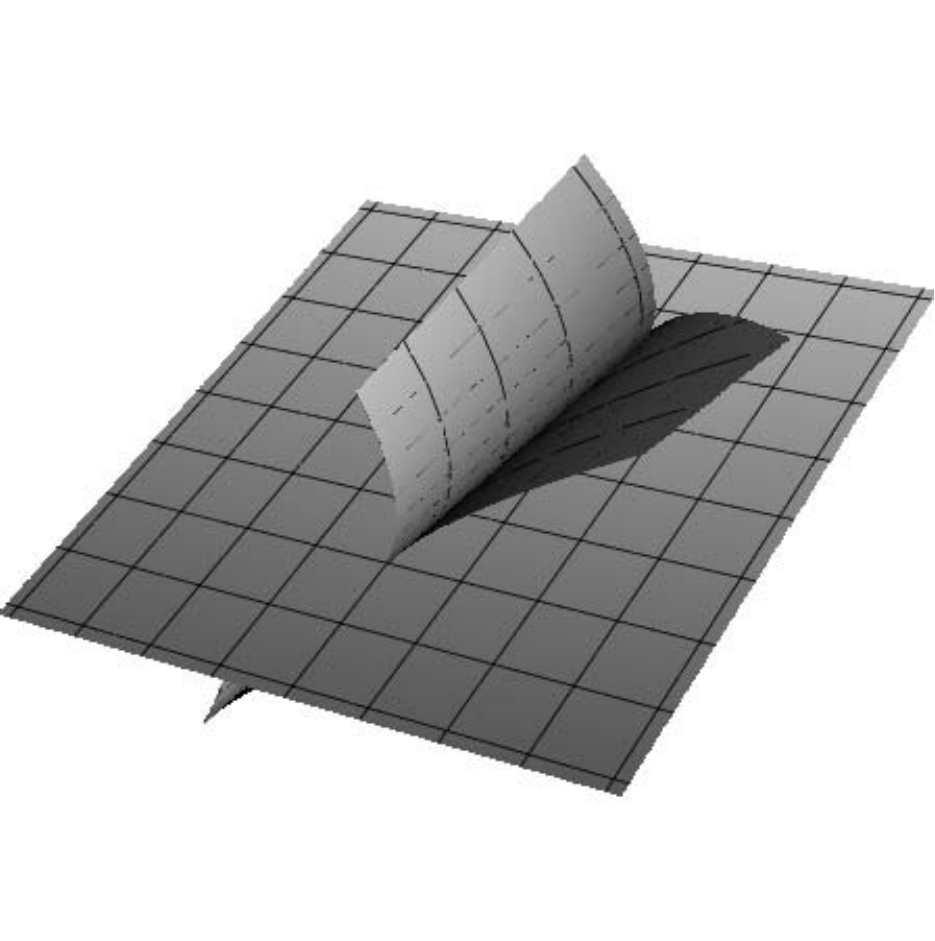}
\hspace{2mm}
\raisebox{0.15\linewidth}[1ex][1ex]{$\longrightarrow$}
\hspace{5mm}
\includegraphics[width=0.35\linewidth]{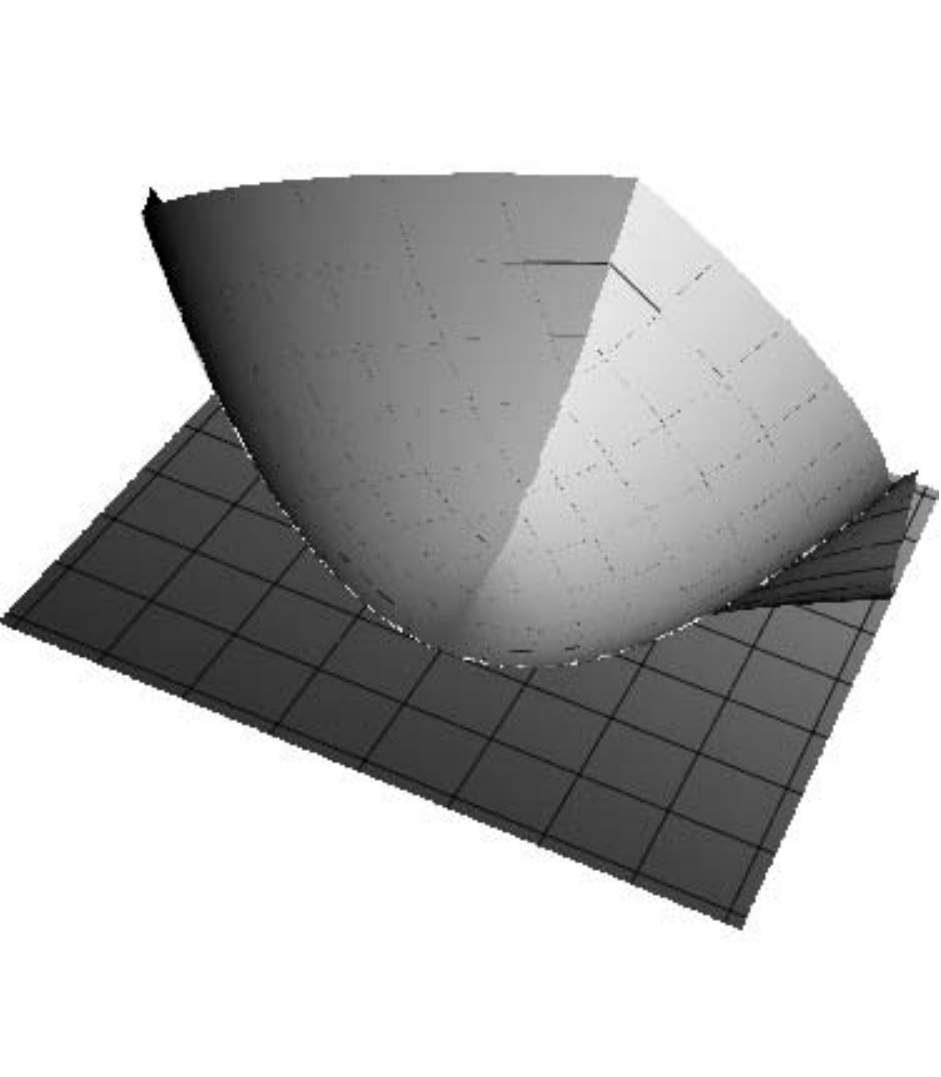}
\caption{A cuspidal cross-cap obtained by folding a cuspidal edge.}
\label{folded}
\end{center}
\end{figure}

We now compute the geometric invariants of
$\phi$.
Since $\phi$ is a cuspidal $S_k$ singularity,
in particular, a singularity of the first kind, we have
the invariants \eqref{eq:invdef1}, \eqref{eq:invdef2}
and \eqref{eq:invdef3}.
The singular set of $\phi$ is given by $S(\phi)=\{y=0\}$.
We set $\Sigma=\phi(S(\phi))$.
Then $\Sigma$ is the image of
$\phi(x,0)=(x,a(x),b_0(x)^2)$.
The osculating plane of $\Sigma$ at
the origin is the plane orthogonal to the vector
$(0,-2b_0'(0)^2,a''(0))$.

\begin{prop}\label{cuspinv}
For a cuspidal $S_k$ singularity $\phi$ obtained by \eqref{eq:folding},
we have
\begin{align*}
\kappa_s^\phi(0)&=a''(0),\\
(\kappa_s^\phi)'(0)&=8 b_1(0) b_0'(0)^3+a'''(0),\\
(\kappa_s^\phi)''(0)&=16b_0'(0)^3b_1'(0)
-16b_0'(0)^4a''(0)-3a''(0)^3+a^{(4)}(0)\\
&-8b_1(0)b_0'(0)^2(2b_1(0)a''(0)-7b_0''(0)),\displaybreak[3]\\
\kappa_\nu^\phi(0)&=2 b_0'(0)^2,\\
(\kappa_\nu^\phi)'(0)&=2 b_0'(0) (-2 b_1(0) a''(0)+3 b_0''(0)),\\
(\kappa_\nu^\phi)''(0)&=-32b_1(0)^2b_0'(0)^4
-24b_0'(0)^6-4b_0'(0)^2a''(0)^2+6b_0''(0)^2\\
&-4b_1(0)(a''(0)b_0''(0)+2b_0'(0)a'''(0))
+b_0'(0)(-8b_1'(0)a''(0)+8b_0'''(0)),\displaybreak[3]\\
\kappa_t^\phi(0)&=4 b_1(0) b_0'(0),\\
(\kappa_t^\phi)'(0)&=-2 b_0'(0)^2 a''(0)
+8 b_0'(0) b_1'(0)+4 b_1(0) b_0''(0),\\
(\kappa_t^\phi)''(0)&=-2(64b_1(0)^3b_0'(0)^3-6b_1'(0)b_0''(0)
\\&+b_0'(0)(6a''(0)b_0''(0)-6b_1''(0)+b_0'(0)a'''(0))\\
&+b_1(0)(32b_0'(0)^5-4b_0'(0)a''(0)^2-2b_0'''(0))),\displaybreak[3]\\
\kappa_c^\phi(0)&=0,\\
(\kappa_c^\phi)'(0)&=12 b_2(0) b_0'(0),\\
(\kappa_c^\phi)''(0)&=12(2b_0'(0)b_2'(0)+b_2(0)b_0''(0)).
\end{align*}
\end{prop}
\begin{proof}
By a straightforward calculation,
we see that
$\kappa_s^\phi(x)$, $\kappa_\nu^\phi(x)$,
$\kappa_t^\phi(x)$ and $\kappa_c^\phi(x)$ are equal to
\begin{align*}
\kappa_s^\phi(x)=&\dfrac{1}{A^{3/2}B^{1/2}}\Bigg(
4 b_0(x)^2 \Big(b_0'(x) (b_0'(x) a''(x)-a'(x) b_0''(x))\\
&\hspace{10mm}+2 b_1(x)
\big(-a'(x) b_0'(x) a''(x)+b_0''(x)+a'(x)^2 b_0''(x)\big)\Big)\\
&\hspace{10mm}+4 b_0(x) b_0'(x)^2 \big(2 b_1(x) (1+a'(x)^2)-a'(x) b_0'(x)\big)+a''(x)\Bigg),\displaybreak[3]\\
\kappa_\nu^\phi(x)=&
\dfrac{2}{AB^{1/2}}
\Big(b_0'(x)^2+b_0(x) \big(-2 b_1(x) a''(x)+b_0''(x)\big)\Big),
\displaybreak[3]\\
\kappa_t^\phi(x)=
&\dfrac{2}{AB}\Bigg(
2 b_0(x) a'(x)^2 b_1'(x)
+
2 b_0(x)b_1'(x)+8  b_0(x)^3 b_0'(x)^2b_1'(x)\\
&\hspace{3mm}
+16 b_0(x)^3 b_1(x)^2 b_0'(x) a''(x)
-a'(x)b_0'(x)^2-a'(x)b_0(x) b_0''(x)\\
&\hspace{3mm}+2 b_1(x) \Big(b_0(x) a'(x) a''(x)
+b_0'(x)+a'(x)^2b_0'(x)-4 b_0(x)^3 b_0''(x)b_0'(x)\Big)\Bigg),
\displaybreak[3]\\
\kappa_c^\phi(x)=&
\dfrac{12}{B^{5/4}}
b_0(x) b_2(x)A^{3/4},
\end{align*}
where
\begin{align*}
A=&1+a'(x)^2+4 b_0(x)^2 b_0'(x)^2,\\
B=&1+16 b_0(x)^2 b_1(x)^2+4 b_0(x)^2 (-2 b_1(x) a'(x)+b_0'(x))^2.
\end{align*}
The rest follows by direct computation.
\end{proof}
We remark that similarly we can obtain higher derivatives of
the invariants but we omit them here.
One can also consider the curvature $\kappa_{\Sigma}$
and the torsion $\tau_{\Sigma}$ of the cuspidal edge curve as
a regular space curve $\Sigma$. These are given by
\begin{equation}\label{eq:ksigmatsigma}
\kappa_{\Sigma}(0)=\sqrt{a_0''(0)^2+4b_0'(0)^4},\quad
\tau_{\Sigma}(0)
=
\frac{2b_0'(0)(3a''(0)b_0''(0)-a'''(0)b_0'(0))}{a_0''(0)^2+4b_0'(0)^4}.
\end{equation}
These two invariants are related to the ones above.
In fact, we have
$$
\kappa_{\Sigma}^2=(\kappa_s^\phi)^2+(\kappa_{\nu}^\phi)^2,\quad
\tau_{\Sigma}=\frac{\kappa_s^\phi(\kappa_{\nu}^\phi)'
-(\kappa_s^\phi)'\kappa_{\nu}^\phi}{(\kappa_s^\phi)^2
+(\kappa_{\nu}^\phi)^2}+\kappa_t^\phi.
$$

Now we assume that
$\phi$ is a cuspidal cross-cap.
Then there is the double point curve.
Here we calculate its invariants.
To calculate the double point curve
take $(x_1,y_1)$ and $(x_2,y_2)$ such that
$\phi(x_1,y_1)=\phi(x_2,y_2)$.
From the first and second components
we get that $x_1=x_2$ and $y_1=-y_2$.
By analysing the equality in the third component, the
double point curve of $\phi$ can be parameterized as
$$
d(y)=\left(\dfrac{d_2}{2}y^2+\dfrac{d_4}{4!}y^4+O(5),y\right),
$$
where $O(n)$ stands for the terms whose degrees are greater than or
equal to $n$, and
$$
d_2 = -\dfrac{2 b_1(0)}{b_0'(0)},\
d_4 = -\dfrac{12}{b_0'(0)^3}
\Big(2 b_3(0, 0) b_0'(0)^2 +
      b_1(0) (-2 b_0'(0) b_1'(0) +
         b_1(0) b_0''(0))\Big).
$$
We set
$$
\tilde d(y)=f(d(y)),\qquad
\widehat d(y)=g\circ f(d(y)).
$$
Then
$
\tilde d''=
(-2 b_1(0)/b_0'(0),1,0)$,
$
\tilde d'''=
(0,0,6 b_2(0))$,
and
$$
\tilde d^{(4)}=
12\left(-
\dfrac{2 b_3(0) b_0'(0)^2
-2 b_0'(0) b_1'(0)b_1(0) +b_1(0)^2 b_0''(0)}
{b_0'(0)^3},
\dfrac{b_1(0)^2 a_0''(0)}{b_0'(0)^2},0\right).
$$
Thus $\tilde d$ at $0$ is of $(2,3)$-type.
The invariants defined in Section \ref{sec:spcurv} are
\begin{align*}
\kappa_{sing}(\tilde d)
=&
6|b_2(0)|
\left(1+\dfrac{4 b_1(0)^2}{b_0'(0)^2}\right)^{3/4},\\
\tau_{sing}(\tilde d)
=&
-2\left(1+\dfrac{4 b_1(0)^2}{b_0'(0)^2}\right)^{1/2}\\
&\hspace{8mm}
\dfrac{2 b_3(0) b_0'(0)^2
-2 b_0'(0) b_1'(0)b_1(0) -2 b_1(0)^3 a_0''(0)+b_1(0)^2 b_0''(0)}
{b_2(0) b_0'(0) (4 b_1(0)^2+b_0'(0)^2)},
\end{align*}
and $\sigma_{sing}=0$.
We remark that these
invariants are geometric invariants of a singular space curve on a cuspidal edge, but they are also geometric invariants of the cuspidal cross-cap obtained by folding that cuspidal edge.

On the other hand,
$\widehat{d}'(0)=\widehat{d}'''(0)=0$, and
\begin{align}
\widehat{d}''(0)=&\left(-\frac{2b_1(0)}{b_0'(0)},1,0\right),
\label{eq:dhatpp}\\
\widehat{d}^{(4)}(0)
=&\left(
\frac{-2b_3(0)b_0'(0)^2+2b_0'(0)b_1'(0)b_1(0)-b_1(0)^2b_0''(0))}{b_0'(0)^3},
\frac{b_1(0)^2a_0''(0)}{b_0'(0)^2},0\right).
\label{eq:dhatpppp}
\end{align}
Since $\widehat{d}''(0)\times\widehat{d}^{(4)}(0)\ne0$,
the limiting tangent vector of the curve $\widehat{d}$ can be
considered to be $\widehat{d}''(0)$ and the osculating plane is
generated by $\widehat{d}''(0)$ and $\widehat{d}^{(4)}(0)$.
Moreover,
one can take the limit tending to $y=0$
of the curvature $\kappa_{\widehat{d}}$
and the torsion $\tau_{\widehat{d}}$ of $\widehat{d}$
along a regular space curve on $\{y\ne0\}$ as follows:
\begin{prop}
For a cuspidal cross-cap obtained by folding a cuspidal edge with parametrisation\/ $\phi$ as in Proposition\/ {\rm \ref{cuspinv}}
$$
\lim_{y\to0}\kappa^2(y)
=
\Bigg(
\dfrac{
2 b_3(0) b_0'(0)^2
-2 b_0'(0) b_1'(0)b_1(0) +b_1(0)^2 b_0''(0)
}{
324 (4 b_1(0)^2+b_0'(0)^2)^3}
\Bigg)^2
$$
and
$$\lim_{y\to0}
\tau(y)
=
\dfrac{48 b_2(0)^2 b_0'(0)^3}
{2 b_3(0) b_0'(0)^2
-2 b_0'(0) b_1'(0)b_1(0) +b_1(0)^2 b_0''(0)}.
$$
\end{prop}
\begin{proof}
Let $\kappa$, $\tau$ be the curvature and the torsion
of $\widehat d$ as a space curve.
By applying l'H\^opital's rule 6 times,
we have
\begin{align*}
\lim_{y\to0}
\kappa^2(y)
&=
\lim_{y\to0}\dfrac{(\hat d'(y)\times \hat d''(y))
\cdot(\hat d'(y)\times \hat d''(y))}
{(\hat d'(y)\cdot\hat d'(y))^3}\\
&=
\dfrac{(\hat d''(0)\times \hat d^{(4)}(0))\cdot
(\hat d''(0)\times \hat d^{(4)}(0))}
{36(\hat d''(0)\cdot\hat d''(0))^3}\\
\lim_{y\to0}
\tau(y)
&=
\lim_{y\to0}
\dfrac{\det(\hat d'(y),\hat d''(y),\hat d'''(y))}
{(\hat d'(y)\times \hat d''(y))
\cdot(\hat d'(y)\times \hat d''(y))}\\
&=
\dfrac{4\det(\hat d''(y),\hat d^{(4)}(y),\hat d^{(6)}(y))}
{5(\hat d''(0)\times \hat d^{(4)}(0))\cdot
(\hat d''(0)\times \hat d^{(4)}(0))}.
\end{align*}
\end{proof}
We can obtain in a similar way the limits
of the geodesic and normal curvature of the double point
curve as a curve on the cuspidal cross-cap surface,
but they are related to the previous invariants and we will omit them here.

Since $\phi$ is not a cuspidal edge,
the invariants $B$ and $\kappa_c^{{\rm r}}$ can be defined
by \eqref{eq:bias} and \eqref{eq:kcr} respectively.
Since $(\partial_x,\partial_y)$ is an adapted pair of vector fields,
we obtain
$$
B=24 b_1(0)^2,\qquad
\kappa_c^{{\rm r}}=720 b_1(0)b_2(0).
$$

\subsection{Geometric invariants up to order 5}

Let
$
\phi(x,y)
$
be a cuspidal cross-cap obtained by \eqref{eq:folding}.
In the previous subsections we have obtained the values of different geometric invariants of the cuspidal cross-cap in terms of the coefficients of the generic cuspidal edge. Most of these invariants are independent of each other and in fact determine all the coefficients of the folded cuspidal edge up to order 5.

\begin{theo}
Let $h_1,h_2:(\mathbb R^2,0)\rightarrow(\mathbb R^3,0)$ be map-germs with a cuspidal cross-cap at 0 obtained by folding two generic cuspidal edges. Suppose that the following 16 invariants are the same at 0: $\kappa_s^{\phi},\kappa_{\nu}^\phi,\kappa_t^\phi$, $(\kappa_s^{\phi})',(\kappa_{\nu}^\phi)',(\kappa_t^\phi)',(\kappa_c^\phi)'$, $(\kappa_s^{\phi})'',(\kappa_{\nu}^\phi)'',(\kappa_t^\phi)'',(\kappa_c^\phi)''$, $(\kappa_s^{\phi})''',(\kappa_{\nu}^\phi)'''$, $B,\kappa_c^{{\rm r}}$ and $\tau_{sing}(\tilde d)$. Then there exists a germ of diffeomorphism $\varphi:(\mathbb R^2,0)\rightarrow(\mathbb R^2,0)$ and a germ of an isometry $\Phi:(\mathbb R^3,0)\rightarrow(\mathbb R^3,0)$ such that $$h_1(u,v)-\Phi(h_2(\varphi(u,v)))\in O(6),$$ where $O(6)=\{h:(\mathbb R^2,0)\rightarrow(\mathbb R^3,0):j^5h=0\}.$
\end{theo}
\begin{proof}
By expanding $\phi(x,y)$
up to order 5 we get
\begin{align*}
&\Bigg(x,\frac{a''(0)}{2}x^2+\frac{a'''(0)}{6}x^3+\frac{a^{(4)}(0)}{24}x^4+\frac{a^{(5)}(0)}{120}x^5+\frac{y^2}{2},b_0'(0)^2x^2+b_0'(0)b_0''(0)x^3\\
&+2b_0'(0)b_1(0)xy^2+(\frac{1}{3}b_0'(0)b_0'''(0)+\frac{1}{4}b_0''(0)^2)x^4+(2b_0'(0)b_1'(0)+b_0''(0)b_1(0))x^2y^2\\
&+2b_0'(0)b_2(0)xy^3+b_1(0)^2y^4+(\frac{1}{12}b_0'(0)b_0^{(4)}(0)+\frac{1}{3}b_0''(0)b_0'''(0))x^5\\
&+(b_0'(0)b_1''(0)+b_0''(0)b_1'(0)+\frac{1}{3}b_0'''(0)b_1(0))x^3y^2+(2b_0'(0)b_2'(0)+b_0''(0)b_2(0))x^2y^3\\
&+2(b_1(0)b_1'(0)+b_0'(0)b_3(0))xy^4+2b_1(0)b_2(0)y^5 \Bigg).
\end{align*}
Rewriting this expression as
$$\left(
x,\sum_{i=2}^5f_ix^i+\frac{y^2}{2},\sum_{2\leq i+j\leq
5}g_{ij}x^iy^j\right),
$$
we get the following relations between the
coefficients and the geometric invariants:
\begin{align*}
f_2&=\frac{1}{2}\kappa_s^\phi,\\
f_3&=\frac{1}{6}((\kappa_s^\phi)'-\kappa_{\nu}^{\phi}\kappa_t^\phi),\\
f_4&=\frac{1}{24}((\kappa_s^\phi)''-\kappa_{\nu}^{\phi}(\kappa_t^\phi)'-2(\kappa_{\nu}^{\phi})'\kappa_t^\phi+3(\kappa_{\nu}^\phi)^2\kappa_s^\phi+3(\kappa_s^{\phi})^3),\\
f_5&=\frac{1}{120}((\kappa_s^\phi)'''-\varphi_1),
\end{align*}
where $\varphi_1$ is a function of
$\kappa_s^\phi,\kappa_{\nu}^\phi,\kappa_t^\phi$ and their
derivatives up to order 2,
\begin{align*}
g_{20}&=\frac{1}{2}\kappa_{\nu}^\phi,\\
g_{30}&=\frac{1}{6}((\kappa_{\nu}^\phi)'+\kappa_t^\phi\kappa_s^\phi),\\
g_{40}&=\frac{1}{24}((\kappa_{\nu}^\phi)''+(\kappa_t^\phi)'\kappa_s^\phi+2\kappa_t^\phi(\kappa_s^\phi)'+3\kappa_{\nu}^\phi(\kappa_s^\phi)^2+(\kappa_{\nu}^\phi)^3-3(\kappa_t^\phi)^2\kappa_{\nu}^\phi),\\
g_{50}&=\frac{1}{120}((\kappa_{\nu}^\phi)'''-\varphi_2),
\end{align*}
where $\varphi_2$ is a function of
$\kappa_s^\phi,\kappa_{\nu}^\phi,\kappa_t^\phi$ and their
derivatives up to order 2, and
\begin{align*}
g_{12}&=\frac{1}{2}\kappa_t^\phi,\\
g_{22}&=\frac{1}{4}((\kappa_t^\phi)'+\kappa_s^\phi\kappa_{\nu}^\phi),\\
g_{32}&=\frac{1}{12}((\kappa_t^\phi)''+(\kappa_s^\phi)'\kappa_{\nu}^\phi+2\kappa_s^\phi(\kappa_{\nu}^\phi)'+2(\kappa_t^\phi)^3+4\kappa_t^\phi(\kappa_{\nu}^\phi)^2),\\
g_{13}&=\frac{1}{6}(\kappa_c^\phi)',\\
g_{23}&=\frac{1}{12}(\kappa_c^\phi)'',\\
g_{04}&=\frac{1}{24}B,\\
g_{05}&=\frac{1}{360}\kappa_c^{{\rm r}},\\
g_{14}&=\varphi_3,
\end{align*}
where $\varphi_3$ is a function of
$\kappa_s^\phi,\kappa_{\nu}^\phi,\kappa_t^\phi,(\kappa_c^\phi)',(\kappa_{\nu}^\phi)',(\kappa_t^\phi)',B,\kappa_c^{{\rm r}}$
and $\tau_{sing}(\widetilde d)$.
\end{proof}

\section{Flat geometry: contact with planes}

In this section we study the contact of the cuspidal cross-cap with planes.
Instead of analysing the different types of contact a
generic cuspidal cross-cap can have with a fixed plane, we fix a
model of the cuspidal cross-cap and study the contact with the zero
fibres of submersions. We then relate the singularities of the
height functions with the geometrical invariants studied in the
previous section.

\subsection{Submersions on the cuspidal cross-cap}

Given the $\mathcal A$-normal form $f(x,y)=(x,y^2,xy^3)$ of a
cuspidal cross-cap (or folded umbrella), we classify germs of
submersions $g:\mathbb R^3,0\to \mathbb R,0$ up to $\mathcal
R(X)$-equivalence, with $X=f(\mathbb R^2,0)$.
Here, for a subset germ $X,0\subset \mathbb R^3,0$,
two map germs $g_1,g_2:\mathbb R^3,0\to \mathbb R,0$ are
$\mathcal R(X)$-equivalent if
there exists a diffeomorphism germ
$k:\mathbb R^3,0\to \mathbb R^3,0$
satisfying $k(X)=X$ and $g_1\circ k=g_2$. Following \cite{osettari2} we use the finer $\mathcal R(X)$-equivalence instead of $\mathcal K(X)$-equivalence.
The defining equation
of $X$ is given by $h(u,v,w)=w^2-u^2v^3$.

Let $\Theta(X)$ be the $\mathcal{E}_3$-module of vector fields in
$\mathbb R^3$ tangent to $X$ (called $Derlog(X)$ in other texts),
where
$\mathcal E_3$ is the ring of germs of real functions in 3 variables
and $\mathcal M_3$ is its maximal ideal. We have $\xi\in\Theta(X)$
if and only if $\xi h=\lambda h$ for some function $\lambda$
(\cite{brucewest}).

\begin{prop}\label{prop:genO(X)}
The $\mathcal{E}_3$-module $\Theta(X) $ of vector fields in $\mathbb
R^3$ tangent to $X$ is generated by the vector fields $
\xi_1=3u\frac{\partial}{\partial u}-2v\frac{\partial}{\partial v},$
$\xi_2=2v\frac{\partial}{\partial v}+3w\frac{\partial}{\partial w},$
$\xi_3=2w\frac{\partial}{\partial v}+3u^2v^2\frac{\partial}{\partial
w},$ $\xi_4=w\frac{\partial}{\partial
u}+uv^3\frac{\partial}{\partial w}. $
\end{prop}

\begin{proof}
Since $h$ is quasihomogeneous $\Theta(X)=\langle \xi_e\rangle\oplus
\Theta_0(X)$, where
$\xi_e=\frac{1}{3}(\xi_1+4\xi_2)=u\frac{\partial}{\partial
u}+2v\frac{\partial}{\partial v}+4w\frac{\partial}{\partial w}$ is
the Euler vector field and $\Theta_0(X)$ are vector fields such that
$\xi(h)=0$ (see \cite{bruceroberts}). It can be seen that $\xi_1$,
$\xi_3$ and $\xi_4$ generate the kernel of the map $\Phi:\mathcal
O_3^3\rightarrow\mathbb R$ given by
$\Phi(\xi)=\sum_{i=1}^3\xi^i\frac{\partial h}{\partial x_i}$ for
$\xi=(\xi^1,\xi^2,\xi^3)$.
\end{proof}

Let $\Theta_1(X)=\{\delta\in\Theta(X):j^1\delta=0\}$. It follows
from Proposition \ref{prop:genO(X)} that
$$
\Theta_1(X)=\mathcal M_3.\{ \mathcal \xi_1,\xi_2,\xi_3,\xi_4\}.
$$

For $f\in \mathcal E_3$, we define  $\Theta(X){\cdot}
f=\{\eta(f)\,|\, \eta\in \Theta(X)\}$. We define similarly
$\Theta_1(X){\cdot} f$ and the following tangent spaces to the
$\mathcal R(X)$-orbit of $f$ at the germ $f$:
$$
L\mathcal R_1(X){\cdot}f=\Theta_1(X){\cdot} f,\quad L\mathcal
R(X){\cdot}f=L_e{\cal R}(X){\cdot}f=\Theta(X){\cdot} f.
$$

The ${\mathcal{R}(X)}$-codimension of $f$  is given by
$d(f,{\mathcal{R}(X)})=\dim_{\mathbb{R}}({\cal
M}_3/L{\mathcal{R}(X)}(f))\,.$

The listing of representatives of the orbits (i.e., the
classification) of $\mathcal R(X)$-finitely determined germs is
carried out inductively on the jet level. The method used here is
that of the complete transversal \cite{bkd} adapted for the
$\mathcal R(X)$-action. We have the following result which is a
version of Theorem 3.11 in \cite{brucewest} for the group $\mathcal
R(X)$.

\begin{prop}\label{prop:completeTrans}
Let $f:\mathbb R^3,0\to \mathbb R,0$ be a smooth germ and
$h_1,\ldots, h_r$ be homogeneous polynomials of degree $k+1$ with
the property that
$$
\mathcal{M}_3^{k+1} \subset L\mathcal R_1(X){\cdot}f + sp\{
h_1,\ldots, h_r\} +\mathcal{M}_3^{k+2}.
$$
Then any germ $g$ with $j^kg(0)=j^kf(0)$ is $\mathcal
R_1(X)$-equivalent to a germ of the form $f(x)+\sum_{i=1}^l
u_ih_i(x)+\phi(x) $, where $\phi(x)\in \mathcal M_n^{k+2}$. The
vector subspace $sp\{ h_1,\ldots, h_r\}$ is called a complete
$(k+1)$-$\mathcal R(X)$-transversal of $f$.
\end{prop}

\begin{cor}\label{cor:detrminacy} If $ \mathcal{M}_3^{k+1}\subset L\mathcal R_1(X){\cdot}f
+\mathcal{M}_3^{k+2}$ then $f$ is $k-\mathcal{R}(X)$-determined.
\end{cor}

We also need the following  result about trivial families.

 \begin{prop}{\rm (\cite{brucewest})} \label{prop:trivialfam}
Let $F:\mathbb R^3\times \mathbb R,(0,0)\to \mathbb R,0$ be a smooth
family of functions with $F(0,t)=0$ for $t$ small. Let
$\xi_1,\ldots,\xi_p$ be vector fields in $\Theta(X)$ vanishing at
$0\in \mathbb R^3$. Then the family $F$ is $k-\mathcal R(X)$-trivial
if $\frac{\partial F}{\partial t}\in \mathcal
E_{4}.\{\xi_1(F),\ldots,\xi_p(F)\} +\mathcal M^{k+1}_3. $
\end{prop}

Two families of germs of functions $F$ and $G:(\mathbb R^3 \times
\mathbb R^l,(0,0))\to (\mathbb{R},0)$  are
\mbox{$P$-$\mathcal{R}^+(X)$}-equivalent if there exist a germ of a
diffeomorphism $\Phi:(\mathbb R^3 \times \mathbb R^l,(0,0)) \to
(\mathbb R^3 \times \mathbb R^l,(0,0))$ preserving $(X\times \mathbb
R^l,(0,0))$ and of the form $\Phi(x,u)=(\alpha(x,u),\psi(u))$ and a
germ of a function $c:(\mathbb R^l,0) \to\mathbb R$ such that $
G(x,u) = F(\Phi(x,u)) + c(u). $

A family $F$ is said to be an $\mathcal{R}^+(X)$-versal deformation
of $F_0(x)=F(x,0)$ if any other deformation $G$ of $F_0$ can be
written in the form  $G(x,u) = F(\Phi(x,u)) + c(u)$ for some germs
of smooth mappings $\Phi$ and $c$ as above with $\Phi$ not
necessarily a germ of diffeomorphism.

Given a family of germs of functions $F$, we write
$\dot{F}_i(x)=\frac{\partial F}{\partial u_i}(x,0).$

\begin{prop}\label{theo:InfcondUnfFunc}
A deformation $F:(\mathbb R^3 \times \mathbb R^l,(0,0))\to
(\mathbb{R},0)$ of a germ of a function $f$ on $X$ is
$\mathcal{R}^+(X)$-versal if and only if
$$
L\mathcal R_e(X)\cdot{}f + \mathbb R.\left\{1,
\dot{F}_1,\ldots,\dot{F}_l\right\}= \mathcal{E}_3.
$$
\end{prop}

We can now state the result about the  $\mathcal
R(X)$-classification of germs of submersions.

\begin{theo} \label{theo:Classification}
Let $X$ be the germ of the $\mathcal A$-model of the cuspidal
cross-cap parametrised by $f(x,y,z)=(x,y^2,xy^3)$. Denote by
$(u,v,w)$ the coordinates in the target. Then any $\mathcal
R(X)$-finitely determined germ of a submersion in $\mathcal M_3$
with $\mathcal R(X)$-codimension $\le 2$ (of the stratum in the
presence of moduli) is $\mathcal R(X)$-equivalent to one of the
germs in {\rm Table \ref{tab:germsubm}}.

\begin{table}[ht]
\caption{Germs of submersions in $\mathcal M_3$ of $\mathcal
R(X)$-codimension $\le 2$.}
\begin{center}
{\begin{tabular}{lcl}
\hline
Normal form & $d(f,\mathcal{R}(X))$ &$\mathcal{R}^+(X)$-versal deformation\\
\hline
$u\pm v$ & $0$ &$u\pm v$\\
$u\pm v^2$ & $1$&$u\pm v^2+a_1v$\\
$u\pm v^3$& $2$&$u\pm v^3+a_1v+a_2v^2$\\
$\pm v\pm u^2$& $1$&$\pm v\pm u^2+a_1u$\\
$\pm v+ u^3$& $2$&$\pm v+u^3+a_1u+a_2u^2$\\
$w\pm u^2+buv+cu^2$, $c\ne 0,\frac{b^2}{4}$ &$2^{(*)}$&
$w\pm u^2+buv+cu^2+a_1u+a_2v$\\
\hline
\end{tabular}
}
\\
{\footnotesize $(*)$: $b,c$ are moduli and the codimension is that
of the stratum.}
\end{center}
\label{tab:germsubm}
\end{table}

\end{theo}

\begin{proof} The linear changes of coordinates in $\mathcal R(X)$ obtained
by integrating the 1-jets of the vector fields in $\Theta(X)$ are
$$
\begin{array}{cc}
\begin{array}{rcl}
\eta_1(u,v,w)&=&(e^{3\alpha}u,e^{-2\alpha}v,w)\\
\eta_2(u,v,w)&=&(u,e^{2\alpha}v,e^{3\alpha}w)\\
\eta_3(u,v,w)&=&(u,v+\alpha w,w)\\
\eta_4(u,v,w)&=&(u+\alpha w,v,w)
\end{array},
&\alpha,\beta\in \mathbb R
\end{array}
$$


Consider a non-zero 1-jet $g=au+bv+cw$ (we are interested in
submersions). If $a\neq 0$, by $\eta_4$ we can make $c=0$. If $a=0$
and $b\neq 0$, we use $\eta_3$ to set $c=0$. We can also use the changes $(u,v,w)\mapsto(-u,v,w)$ and $(u,v,w)\mapsto(u,v,-w)$, which preserve $X$, so the orbits in the 1-jet space are $u,u\pm
v,\pm v,w$.

\smallskip
$\bullet$ Consider the 1-jet $g=u\pm v$. Then $\xi_1(g)=3u\mp2v$,
$\xi_2(g)=\pm2v$,
 $\xi_3(g)=2w$ and $\xi_4(g)=w$, so it is
1-determined and has codimension 0.

\smallskip
$\bullet$ Consider the 1-jet $g=u$. Then $\xi_1(g)=3u$,
$\xi_2(g)=0$, $\xi_3(g)=0$ and $\xi_4(g)=w$, so $\mathcal
M_{3}^l\subset L\mathcal R_1(X)\cdot{}g + sp\{v^l\} + \mathcal
M_{3}^{l+1}$, that is, a complete $l$-transversal is given by
$g=u+av^l$ ($l\ge 2$). Using $\eta_1$ and multiplication by
constants we can fix $a=\pm 1$ We have
$$
\begin{array}{rcl}
\xi_1(g)&=&3u\mp 2lv^l,\\
\xi_2(g)&=&\pm 2v^l,\\
\xi_3(g)&=&2lv^{l-1}w\\
\xi_4(g)&=&w.
\end{array}
$$

Now $\mathcal M_{3}^{l+1}\subset L\mathcal R_1(X)\cdot{}g + \mathcal
M_{3}^{l+2}$, so $u\pm v^l$ is $l$-determined and has codimension
$l-1$.

\smallskip
$\bullet$ Consider the 1-jet $g=\pm v$. Then $\xi_1(g)=-2v$,
$\xi_2(g)=2v$, $\xi_3(g)=2w$ and $\xi_4(g)=0$, so $\mathcal
M_{3}^l\subset L\mathcal R_1(X)\cdot{}g + sp\{u^l\} + \mathcal
M_{3}^{l+1}$, that is, a complete $l$-transversal is given by
$g=v+au^l$ ($l\ge 2$). Using $\eta_2$ and multiplication by
constants we can fix $a=\pm 1$ We have
$$
\begin{array}{rcl}
\xi_1(g)&=&\pm3lu^l\mp 2v,\\
\xi_2(g)&=&\pm 2v,\\
\xi_3(g)&=&\pm 2w\\
\xi_4(g)&=&\pm lwu^{l-1}.
\end{array}
$$

Now $\mathcal M_{3}^{l+1}\subset L\mathcal R_1(X)\cdot{}g + \mathcal
M_{3}^{l+2}$, so $\pm v\pm u^l$ is $l$-determined and has codimension
$l-1$.

\smallskip
$\bullet$ Consider the 1-jet $g=w$. Then $\xi_1(g)=0$,
$\xi_2(g)=3w$, $\xi_3(g)=3u^2v^2$ and $\xi_4(g)=uv^3$, and so a
complete 2-transversal is $g=w+au^2+buv+cv^2$. Applying Proposition
\ref{prop:trivialfam} does not show whether $g$ is trivial seen as a
family with parameter $a,b$ or $c$. If $a\neq 0$, chose $\alpha$
such that $ae^{6\alpha}=\pm 1$, then using $\eta_1$ we get
$g=w\pm u^2+b'uv+c'v^2$. Set $b'=b$ and $c'=c$.

Consider the 2-jet $g=w\pm u^2+buv+cv^2$, we have $\mathcal
M_3^3\subset L\mathcal R_1(X)\cdot{}g + sp\{uv^2,v^3\} + \mathcal
M_{3}^{4}$, that is, a complete $3$-transversal is given by
$g=w+u^2+buv+cv^2+duv^2+ev^3$.

Consider $g$ as a 1-parameter family parametrised by $d$. Now,
\begin{align*}
&\langle\xi_1(g),\xi_2(g),\xi_3(g),\xi_4(g)\rangle\\
=&\langle
6u^2+buv-4cv^2-uv^2-6ev^3, 2buv+4cv^2+4duv^2+6ev^3+3w,\\
 &\hspace{10mm}
2buw+4cvw+4duvw+6ev^2w+3u^2v^2, 2uw+bvw+dv^2w+uv^3\rangle.
\end{align*}
From
$\xi_2(g)$ we get everything of order 3 with $w$ mod $\mathcal
M_3^4$. From $\xi_3(g)$ and $\xi_4(g)$, if $c\neq \frac{b^2}{4},0$
we get $uw$ and $vw$. Now, from $u\xi_1(g)$, $v\xi_1(g)$,
$u\xi_2(g)$ and $v\xi_2(g)$ we get (again if $c\neq
\frac{b^2}{4},0$) $u^3,u^2v,uv^2$ and $v^3$. This means that $g$ is
trivial along $d$ and we can set $d=0$. Similarly, since $d\neq 0$
was not a condition for the previous calculations, $g$ is trivial
along $e$ too. In fact, for $g=w+u^2+buv+cv^2$, it can be seen that
if $c\neq \frac{b^2}{4},0$ then $\mathcal M_3^4\subset L\mathcal
R_1(X)\cdot{}g  + \mathcal M_{3}^{5}$, which means that $g$ is
3-determined. It has codimension 4 and the normal space is generated
by $u,v,uv$ and $v^2$.
\end{proof}

\subsection{Height functions: geometrical interpretations}\label{sec:geomint}

We consider the general form of the cuspidal cross-cap constructed
from a cuspidal edge through a folding map.
Let
$$
\phi(x,y)=\left(
x,a(x)+\frac{y^2}{2},(b_0(x)+b_1(x)y^2+b_2(x)y^3+b_3(x,y)y^4)^2\right),$$
be a cuspidal cross-cap obtained by \eqref{eq:folding}
with $a(0)=a'(0)=b_0(0)=0$ and $b_0'(0)b_2(0)\neq 0$.

Recall that the cuspidal edge $\Sigma$ is given by
$\phi(x,0)=(x,a(x),b_0(x)^2)$ and the osculating plane of $\Sigma$
at the origin is the plane orthogonal to the vector
$(0,-2b_0'(0)^2,a''(0))$. We remark that since
$\kappa_{\nu}^{g\circ\phi}(0)=b_0'(0)\neq 0$, the osculating plane
can never coincide with the tangent cone. Recall too that the
torsion is given by
$$\tau_{\Sigma}(0)=\frac{2b_0'(0)(3a''(0)b_0''(0)-a'''(0)b_0'(0))}{a_0''(0)^2+4b_0'(0)^4}.$$

The double point curve is $\widehat{d}(y)=\phi(d(y),y)$ as in the
setting of Section 2. It is a singular curve, the limiting tangent
vector can be considered to be $\widehat{d}''(0)$ and the osculating
plane is generated by $\widehat{d}''(0)$ and
$\widehat{d}^{(4)}(0)$, as in \eqref{eq:dhatpp} and \eqref{eq:dhatpppp}.

The family of height functions $H:M\times S^2\to \mathbb R$ on $M$
is given by $H((x,y),{\bf v})=H_{{\bf v}}(x,y)=\phi(x,y)\cdot {\bf
v}.$ The height function $H_{{\bf v}}$ on $M$ along a fixed
direction ${{\bf v}}$ measures the contact of $M$ at $p$ with the
plane $\pi_{{\bf v}}$ through $p$ and orthogonal to ${{\bf v}}$. The
contact of $M$ with  $\pi_{{\bf v}}$  is described by that of the
fibre $g=0$ with the model cuspidal cross-cap $X$, with $g$ as in
Theorem \ref{theo:Classification}. Following the transversality
theorem in the Appendix of \cite{brucewest}, for a generic cuspidal
cross-cap, the height functions $H_{{\bf v}}$, for any ${{\bf v}}\in
S^2$, can only have singularities of $\mathcal R(X)$-codimension
$\le 2$ (of the stratum) at any point on the cuspidal edge.

We shall take $M$ parametrised by $\phi$ and write ${\bf
v}=(v_1,v_2,v_3)$. Then,
\begin{align*}
&H_{{\bf v}}(x,y)=H((x,y),{\bf v})\\
=&xv_1+\left(a(x)+\frac{1}{2}y^2\right)v_2+(b_0(x)+b_1(x)y^2+b_2(x)y^3+b_3(x,y)y^4)^2v_3.
\end{align*}

The function $H_{{\bf v}}$  is singular at the origin if and only if
$v_1=0$, that is, if and only if the plane $\pi_{{\bf v}}$ contains
the tangential direction to $M$ at the origin.

If the plane $\pi_{{\bf v}}$ is transversal to both the cuspidal
edge and the double point curve, then the contact of $\pi_{{\bf v}}$
with $M$ is described by the contact of the zero fibre of $g=u\pm v$
with the model cuspidal cross-cap $X$.

\begin{prop}
The height function along the double point curve $H_{{\bf
v}}(d(y),y)=\widehat{d}(y)\cdot{\bf v}$ can have
$A_j^{\pm}$-singularities, $j=1,3,5$, which are modeled by the
contact of the zero fibre of the submersions $u\pm v^k$, $k=1,2,3$
resp., with the model cuspidal cross-cap $X$ (i.e. modeled by the composition of the submersions with the parametrisation of the model cuspidal cross-cap along the double point curve). The geometrical
interpretations are as follows:
$$
\begin{array}{rl}
A_1^{\pm}:&\pi_{{\bf v}} \mbox{ \rm is not a tangent plane of }\widehat{d}(y);\\
A_3^{\pm}:&\pi_{{\bf v}} \mbox{ \rm is not the osculating plane of }\widehat{d}(y),\, \tau_{sing}(\tilde d)(0)\ne 0;\\
A_5^{\pm}:&\pi_{{\bf v}} \mbox{ \rm is the osculating plane of
}\widehat{d}(y),\, \tau_{sing}(\tilde d)(0)=0.
\end{array}
$$
\end{prop}
\begin{proof}

The height function along the double point curve $\widehat{d}(y)$ is
always singular.
\begin{align*}
&H_{{\bf v}}(d(y),y)\\
=&d(y)v_1+(a(d(y))+\frac{y^2}{2})v_2
+(b_0(d(y))+b_1(d(y))y^2\\
&+b_2(d(y))y^3+b_3(d(y),y)y^4)^2v_3
\end{align*}
It has an $A_1^{\pm}$-singularity if
$\frac{-b_1(0)}{b_0'(0)}v_1+\frac{1}{2}v_2\neq 0$. This happens when $\pi_{{\bf v}}$ is not a tangent plane to the double point curve and is also described by the case $g=u\pm v$ above.

If $\pi_{{\bf v}}$ is transversal to the cuspidal edge but contains
the limiting tangent vector to the double point curve then the
contact of $\pi_{{\bf v}}$ with $M$ is described by the contact of
the zero fibre of $g=u\pm v^k$, $k=2,3$, with the model cuspidal
cross-cap $X$.  For the height function to have an
$A_3^{\pm}$-singularity the coefficient for $y^4$ must be non-zero,
so
$$\frac{-b_1(0)}{b_0'(0)}v_1+\frac{1}{2}v_2=0
\text{    and    } \widehat{d}^{(4)}_1(0)v_1+\widehat{d}^{(4)}_2(0)v_2\neq 0.$$ This case is described by the case $k=2$.
The second equation means that $\pi_{{\bf v}}$ is not the osculating
plane and can be written as
$$\frac{-2b_3(0)b_0'(0)^2+2b_0'(0)b_1'(0)b_1(0)-b_1(0)^2b_0''(0))}{b_0'(0)^3}v_1+\frac{b_1(0)^2a_0''(0)}{b_0'(0)^2}v_2\neq 0.$$
Substituting $v_2=\frac{2b_1(0)}{b_0'(0)}v_1$ in this equation we
get
$v_1(\frac{1}{b_0'(0)}(2b_1(0)^3a_0''(0)-2b_3(0)b_0'(0)^2+2b_0'(0)b_i'(0)b_1(0)-b_1(0)^2b_0''(0)))\neq
0$, which means that $\tau_{sing}(\tilde d)(0)\neq 0$.

The height function has an $A_5^{\pm}$-singularity if
$$\frac{-b_1(0)}{b_0'(0)}v_1+\frac{1}{2}v_2=0,
\widehat{d}^{(4)}_1(0)v_1+\widehat{d}^{(4)}_2(0)v_2=0 $$
$$\text{    and    } \widehat{d}^{(6)}_1(0)v_1+\widehat{d}^{(6)}_2(0)v_2+720b_2(0)^2v_3\neq 0.$$ Here, $\pi_{{\bf v}}$ is the osculating plane and $\tau_{sing}(\tilde d)(0)=0$ and is described by the case $k=3$.
\end{proof}

\begin{prop}
The height function along the cuspidal edge curve $H_{{\bf v}}(x,0)$
can have $A_1^{\pm}$ or an $A_2$-singularity, which are modeled by
the contact of the zero fibre of the submersions $v\pm u^k$, $k=2,3$
resp., with the model cuspidal cross-cap $X$ (i.e. modeled by the composition of the submersions with the parametrisation of the model cuspidal cross-cap along the cuspidal edge). The geometrical
interpretations are as follows:
$$
\begin{array}{rl}
A_1^{\pm}:&\pi_{{\bf v}} \mbox{ \rm is not the osculating plane of }\Sigma;\\
A_2:&\pi_{{\bf v}} \mbox{ \rm is the osculating plane of
}\widehat{d}(y),\, \tau_{\Sigma}(0)\neq 0.
\end{array}
$$
\end{prop}
\begin{proof}

If $\pi_{{\bf v}}$ is transversal to the double point curve but
contains the tangent vector at the origin to $\Sigma$ and is not the
tangent cone, then the contact of $\pi_{{\bf v}}$ with $M$ is
described by the contact of the zero fibre of $g=v\pm u^k$, $k=2,3$,
with the model cuspidal cross-cap $X$. Here, the height function
along the cuspidal edge is
$$
H_{{\bf
v}}(x,0)=\left(\frac{a''(0)}{2}v_2+b_0'(0)^2v_3\right)x^2+\left(\frac{a'''(0)}{6}v_2+b_0'(0)b_0''(0)v_3\right)x^3+h.o.t.
$$
The height function has an $A_1^{\pm}$-singularity if
$$\frac{a''(0)}{2}v_2+b_0'(0)^2v_3\neq 0$$ (described by the case
$k=2$) and an $A_2$-singularity if
$$\frac{a''(0)}{2}v_2+b_0'(0)^2v_3=0 \text{    and    }
\frac{a'''(0)}{6}v_2+b_0'(0)b_0''(0)v_3\neq 0$$ (described by the
case $k=3$). Geometrically this means that the condition for the
height function to have an $A_1^{\pm}$-singularity is that $\pi_{{\bf v}}$
is not the osculating plane of $\Sigma$. The condition for the
height function to have an $A_2$-singularity is that $\pi_{{\bf v}}$
is the osculating plane and that $a'''(0)b_0'(0)-3a''(0)b_0'(0)\neq
0$, i.e. that $\tau_{\Sigma}(0)\neq 0$.
\end{proof}

The contact of the zero fibre of $g=w\pm u^2+buv+cv^2$ with the
model cuspidal cross-cap $X$ describes the contact of $\pi_{{\bf v}}$ with $M$ when $\pi_{{\bf v}}$ is the tangent cone (but is not
the osculating plane of the cuspidal edge i.e. $\kappa_{\nu}(0)\neq
0$).

For $\phi$ as in \eqref{eq:folding} and ${\bf v}=(0,0,1)$, the height function is
$$
 H_{{\bf v}}(x,y)
=(b_0(x)+b_1(x)y^2+b_2(x)y^3+b_3(x,y)y^4)^2v_3,
$$
which does not yield a finitely determined singularity.

\begin{rem}\label{rema3}
i) Notice that composing $g=w\pm u^2+buv+cv^2$ with the
parametrisation of the model cuspidal cross-cap we get $xy^3\pm
x^2+bxy^2+cy^4$, and since $c\neq 0,\frac{b^2}{4}$ this is an
$A_3^{\pm}$-singularity. So generically there could be an
$A_3^{\pm}$-singularity when $\pi_{{\bf v}}$ is the tangent cone.
One of the reasons for this not to be captured by our approach is
that when we obtain the cuspidal cross-cap as a folding of the
cuspidal edge, all the surface is left on one side of the tangent
cone, whereas the case when the surface is in both sides of the
tangent cone is also generic.


ii) In Theorem 2.11 in \cite{martinsnuno} some conditions for the height function on a corank 1 singular surface to have a corank 2 singularity are given. For the cuspidal cross-cap, those conditions are equivalent to the fact of the tangent cone coinciding with the osculating plane of the cuspidal edge curve. This is not generic, which explains why, generically, the height function on the cuspidal cross-cap only has $A_k$-singularities.
\end{rem}

\subsection{Geometry of functions on the cuspidal cross-cap and
duals}

From the point of view of the geometry of the submersion on the
cuspidal cross-cap and their $\mathcal R^+(X)$-versal deformations,
there are some interesting discriminants that can be studied. Let
$g:\mathbb R^3,0\rightarrow\mathbb R,0$ be a submersion and
$G:\mathbb R^3\times\mathbb R^2\rightarrow\mathbb R$ its
deformation. Now let $F(x,y,a)=G\circ\phi(x,y)=G(x,y^2,xy^3,a)$,
$H(y,a)=F(0,y,a)=G(0,y^2,0,a)$ and $P(x,a)=F(x,0,a)=G(x,0,0,a)$. We
define
$$
\mathscr D_{PD}(G)=\{(a,F(x,y,a))\in\mathbb R^2\times\mathbb
R:\frac{\partial F}{\partial x}=\frac{\partial F}{\partial y}=0
 \text{   at some point   } (x,y,a)\},
$$
$$
\mathscr D_{DPC}(G)=\{(a,H(y,a))\in\mathbb R^2\times\mathbb
R:\frac{\partial H}{\partial y}=0
 \text{   at some point   } (y,a)\},
$$
$$
\mathscr D_{CE}(G)=\{(a,P(x,a))\in\mathbb R^2\times\mathbb
R:\frac{\partial P}{\partial x}=0
 \text{   at some point   } (x,a)\}.
$$

In the above notation $PD$ stands for ``proper dual", $DPC$ for ``double point curve" and $CE$ for ``cuspidal edge". It is not difficult to show that for  two $P$-$\mathcal
R^+(X)$-equivalent deformations $G_1$ and $G_2$ the sets $\mathscr
D_{PD}(G_1)$, $\mathscr D_{DPC}(G_1)$ and $\mathscr
D_{CE}(G_1)$ are diffeomorphic to $\mathscr D_{PD}(G_2)$, $\mathscr D_{DPC}(G_2)$ and $\mathscr D_{CE}(G_2)$, respectively. Therefore, it is enough to
compute the sets $\mathscr D_{PD}(G)$, $\mathscr D_{DPC}(G)$ and
$\mathscr D_{CE}(G)$ for the deformations in Table
\ref{tab:germsubm}.

\medskip
$\bullet$ {\it The germ $g=u\pm v$.}

In this case $\mathscr D_{PD}(G)=\mathscr D_{CE}(G)=\emptyset$ and
$\mathscr D_{DPC}(G)$ is the plane $(a_1,a_2,0)$.

\medskip
$\bullet$ {\it The germs $g=u\pm v^k$, $k=2,3$.}

Here $G(u,v,w,a)=u\pm v^k+a_1v+\ldots+a_{k-1}v^{k-1}$ and
$F(x,y,a)=x\pm y^{2k}+a_1y^2+\ldots a_{k-1}y^{2k-2}$, so $\mathscr
D_{PD}(G)=\mathscr D_{CE}(G)=\emptyset$. On the other hand,
$H(y,a)=\pm y^{2k}+a_1y^2+\ldots a_{k-1}y^{2k-2}$ is the unfolding
of a $B_k$ singularity.

For $k=2$, we have $\frac{\partial H}{\partial y}=4y^3+2a_1y=0$ so
$\mathscr D_{DPC}(G)$ has two components parametrised by
$$(a_1,a_2,0) \text{    and    } (-2y^2,a_2,-y^4).$$
See Figure \ref{uvk} left, where the $a_2$ parameter is not shown.

For $k=3$, we have $\frac{\partial H}{\partial
y}=6y^5+2a_1y+4a_2y^3=0$ so $\mathscr D_{DPC}(G)$ has two components
parametrised by
$$(a_1,a_2,0) \text{    and    } (-3y^4-2a_2y,a_2,-2y^6-a_2y^4).$$
See Figure \ref{uvk} right.

\begin{figure}
\begin{center}
\includegraphics[width=.9\linewidth]{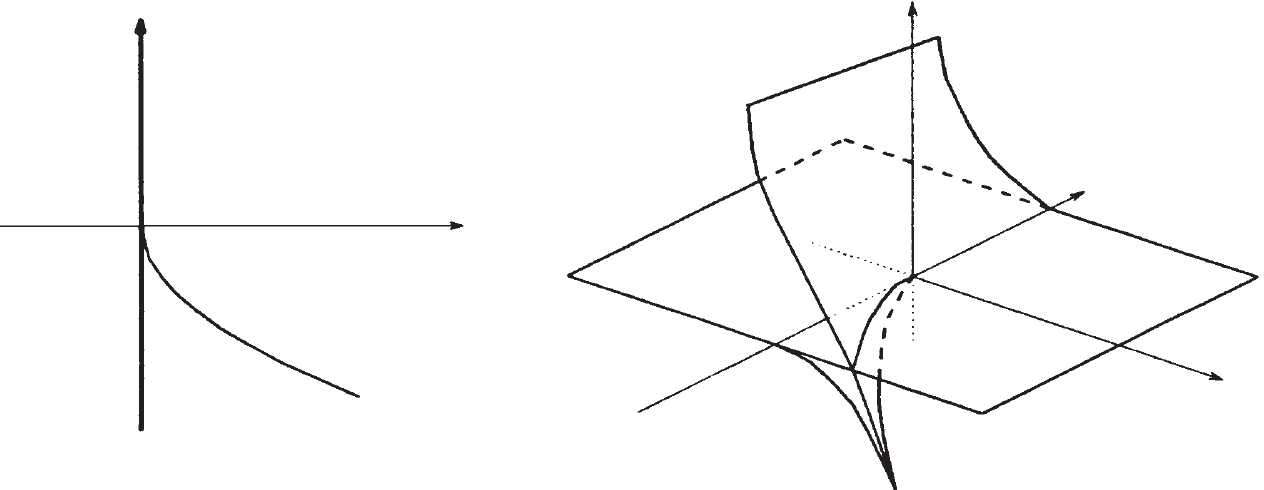}
\caption{These figures correspond to Figures 2a) and 2b) in
\cite{brucewest}.} \label{uvk}
\end{center}
\end{figure}

\medskip
$\bullet$ {\it The germs $g=v\pm u^k$, $k=2,3$.}

Here $G(u,v,w,a)=v\pm u^k+a_1u+\ldots+a_{k-1}u^{k-1}$ and
$F(x,y,a)=y^2\pm x^{k}+a_1x+\ldots a_{k-1}x^{k-1}$, which is the
deformation of an $A_{k-1}$-singularity. $H(y,a)=y^2$ so $\mathscr
D_{DPC}(G)$ is the plane $(a_1,a_2,0)$. Here $\mathscr D_{PD}(G)$
and $\mathscr D_{CE}(G)$ coincide.

For $k=2$, $\frac{\partial F}{\partial x}=\pm 2x+a_1=0$ and
$\frac{\partial F}{\partial y}=2y=0$, so $\mathscr D_{PD}(G)$ is
parametrised by $$(\mp 2x,a_2,\pm x^2).$$ See Figure \ref{vuk} left,
where the parameter $a_2$ is not shown.

For $k=3$, $\frac{\partial F}{\partial x}=\pm 3x^2+a_1+2a_2x=0$ and
$\frac{\partial F}{\partial y}=2y=0$, so $\mathscr D_{PD}(G)$ is a
cuspidal edge parametrised by $$(\mp 3x^2-2a_2x,a_2,\pm
2x^3-a_2x^2).$$ See Figure \ref{vuk} right.

\begin{figure}
\begin{center}
\includegraphics[width=.9\linewidth]{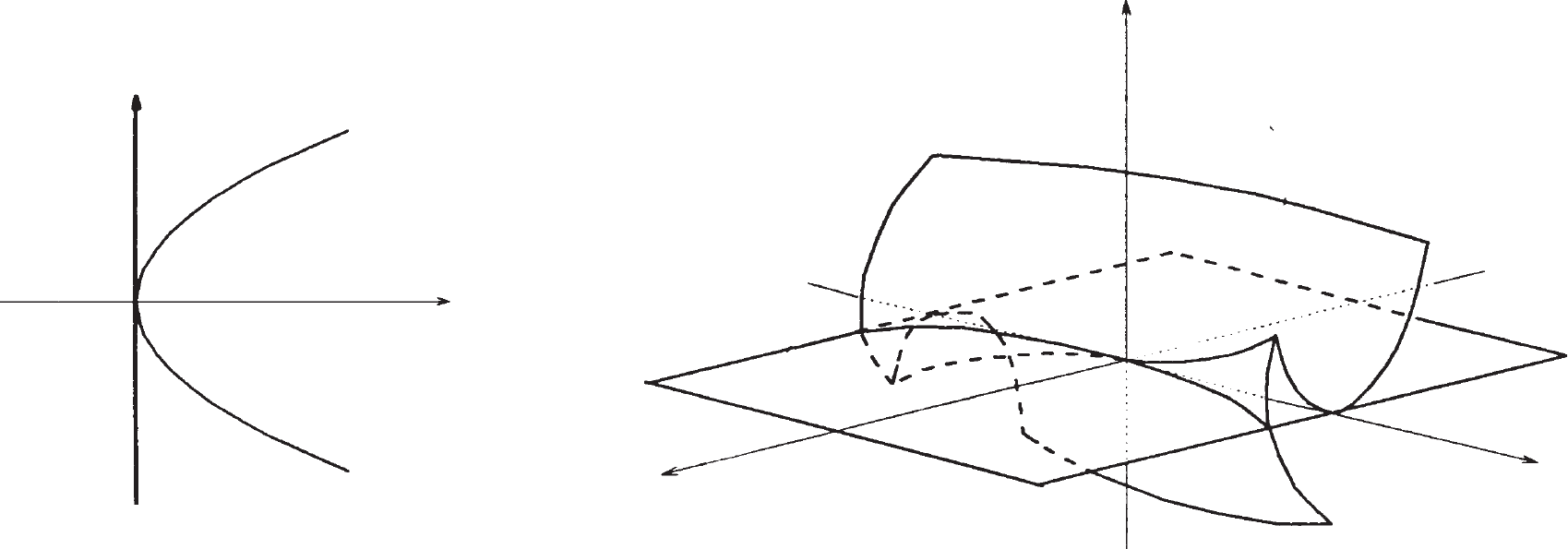}
\caption{These figures correspond to Figures 3 and 4 in
\cite{brucewest}.} \label{vuk}
\end{center}
\end{figure}

\medskip
$\bullet$ {\it The  germ $g=w+u^2+buv+cv^2$.}

Here $G(u,v,w,a)=w+u^2+buv+cv^2+a_1u+a_2v$ and
$F(x,y,a)=xy^3+x^2+bxy^2+cy^4+a_1x+a_2y^2$. $\frac{\partial
F}{\partial x}=y^3+2x+by^2+a_1=0$ and $\frac{\partial F}{\partial
y}=3xy^2+2bxy+4cy^3+2a_2y=0$ which yields two solutions
$$
\left\{
\begin{array}{l}
y=0\\
a_1=-2x
\end{array}
\right.
\quad\text{and}\quad
\left\{
\begin{array}{l}
a_1=-y^3-2x-by^2\\
a_2=-\dfrac{3}{2}xy-bx-2cy^2,
\end{array}
\right.
$$
so $\mathscr D_{PD}(G)$ has two components
parametrised by
$$(-y^3-2x-by^2,-bx-3/2xy-2cy^2,-3/2xy^3-x^2-bxy^2-cy^4),$$ which
is a cuspidal cross-cap, and $$(-2x,a_2,-x^2),$$ which is a
hyperplane containing the cuspidal edge of the first component. The
second component coincides with $\mathscr D_{CE}(G)$, which, in this
case is contained in $\mathscr D_{PD}(G)$ (see Figure \ref{duallast}).

\begin{figure}
\begin{center}
\includegraphics[width=0.5\linewidth]{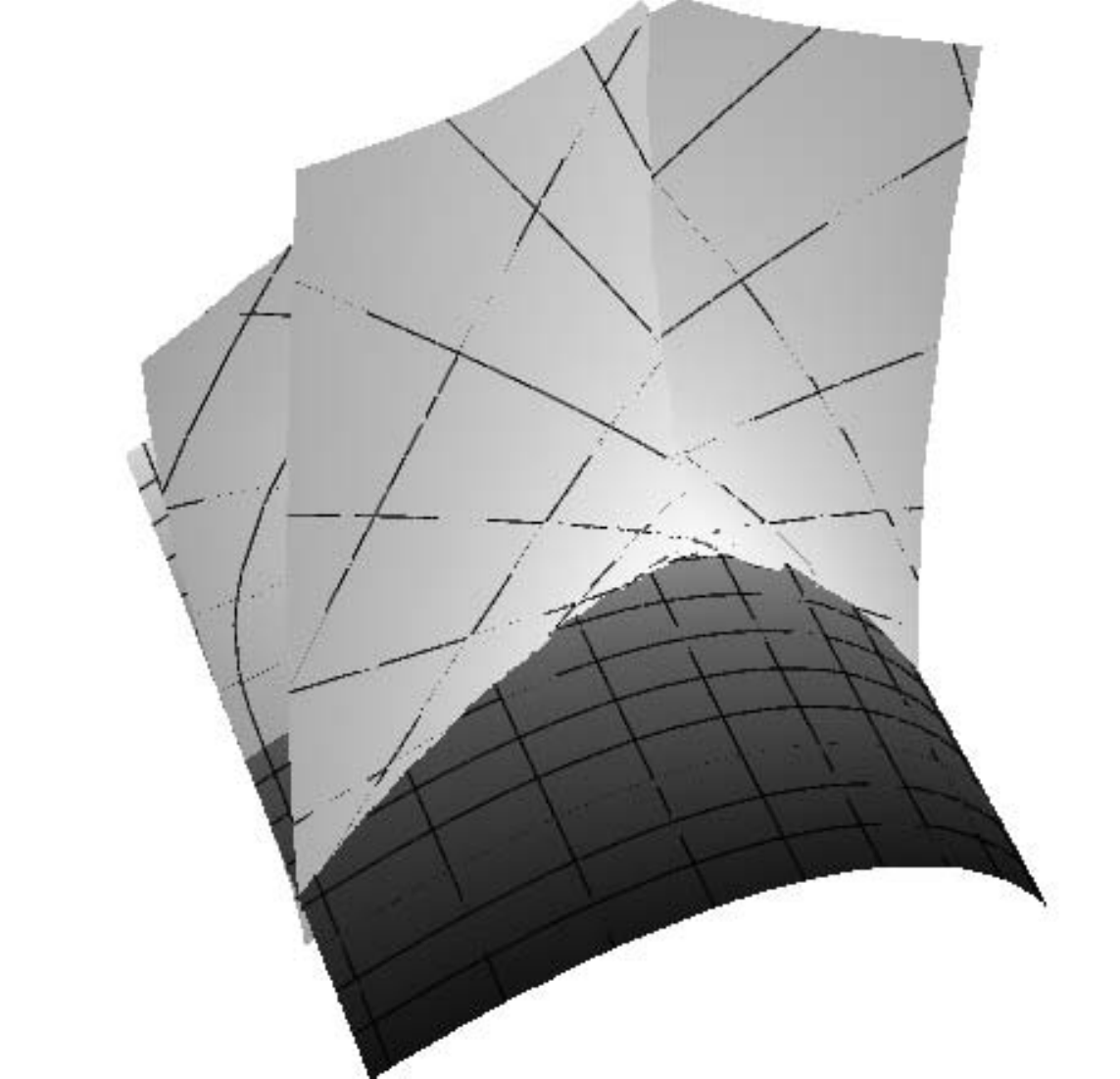}
\caption{The discriminant $\mathscr D_{PD}(G)$. The dark part
corresponds to $\mathscr D_{CE}(G)$.} \label{duallast}
\end{center}
\end{figure}

On the other hand $H(y,a)=cy^4+a_2y^2$, so $\frac{\partial
H}{\partial y}=4cy^3+2a_2y=0$, so $\mathscr D_{DPC}(G)$ consists of
two components, the plane $(a_1,a_2,0)$ and another one parametrised
by $$(a_1,-2cy^2,-cy^4).$$

\bigskip
We can use these discriminants to study the dual of the
cuspidal cross-cap. Consider again the height function $H_{{\bf
v}}(x,y))=\phi(x,y)\cdot \bf v$. We have the sets
$$
\mathscr D_{PD}(H)=\{({{\bf v}},H_{{\bf v}}(x,y))\in  S^2 \times
\mathbb R: \frac{\partial H_{{\bf v}}}{\partial x}=\frac{\partial
H_{{\bf v}}}{\partial y}=0 \mbox{ at } (x,y,{{\bf v}}) \},
$$
$$
\mathscr D_{DPC}(H)=\{({{\bf v}},H_{{\bf v}}(0,y))\in S^2\times
\mathbb R: \frac{\partial H_{{\bf v}}}{\partial y}=0 \mbox{ at  }
(0,y,{{\bf v}}) \}
$$
and
$$
\mathscr D_{CE}(H)=\{({{\bf v}},H_{{\bf v}}(x,0))\in S^2\times
\mathbb R: \frac{\partial H_{{\bf v}}}{\partial x}=0 \mbox{ at  }
(x,0,{{\bf v}}) \}.
$$

If $\pi_{\bf v}$ is a member of the pencil containing the tangential
direction of $M$ but is not the tangent cone to $M$, then the set
$\mathscr D_{PD}(H)$ coincides with $\mathscr D_{CE}(H)$ and
describes locally the dual of the curve $\Sigma$. When $\pi_{\bf v}$
is the tangent cone to $M$, then the set $\mathscr D_{PD}(H)$
consists of two components. One of them is $\mathscr D_{CE}(H)$ (the
dual of $\Sigma$) and the other is the {\it proper dual} of $M$
which is the surface consisting of the tangent planes to $M$ away
from points on $\Sigma$ together with their limits at points on
$\Sigma$, i.e., the tangent cones at points on $\Sigma$. $\mathscr
D_{DPC}(H)$ is the dual of the double point curve. If $\pi_{\bf v}$
is transversal to the limiting tangent vector of the double point
curve, then $\mathscr D_{DPC}(H)$ consists of just one plane
component. If $\pi_{\bf v}$ contains the limiting tangent vector of
the double point curve, then $\mathscr D_{DPC}(H)$ consists of two
components one of which is a plane.

If the contact of $M$ with  $\pi_{{\bf v}}$  is described by that of
the fibre $g=0$ with the model cuspidal edge $X$, with $g$ as in
Theorem \ref{theo:Classification}, then $\mathscr D_{PD}(H)$ (resp.
$\mathscr D_{DPC}(H)$ and $\mathscr D_{CE}(H)$) is diffeomorphic to
$\mathscr D_{PD}(G)$ (resp. $\mathscr D_{DPC}(G)$ and $\mathscr
D_{CE}(G)$), where $G$ is an $\mathcal R^+(X)$-versal deformation of
$g$ with 2-parameters.

\begin{prop}
The previous calculations give the generic models, up to
diffeomorphisms, of $\mathscr D_{PD}(H)$, $\mathscr D_{DPC}(H)$ and
$\mathscr D_{CE}(H)$, and these are as in Figures \ref{uvk},
\ref{vuk} and \ref{duallast}.
\end{prop}

\section{The tangent developable surface of a space curve at a point of zero torsion}

In \cite{cleave}, Cleave proved that the tangent developable surface
of a space curve at a point of zero torsion has a cuspidal cross-cap
singularity. In this section we study the geometry of such a cuspidal cross-cap.

Let $\gamma:I\to \R^3$ be a space curve, and
$\gamma=(\gamma_1,\gamma_2,\gamma_3)$. Without loss of generality we
assume that $\gamma$ is transverse to the plane $X=0$, namely,
$\gamma_1'(0)\ne0$. Then there exists a parameter $u$ such that
$\gamma(u)=(u,\gamma_2(u),\gamma_3(u))$. Let $f$ be the tangent
developable surface of $\gamma$:
$$
f(u,v) = (u+v,\gamma_2(u)+v\gamma_2'(u),\gamma_3(u)+v\gamma_3'(u)).
$$

Thus, if $\det(\tilde\gamma'(0),\tilde\gamma''(0))\ne0$, then $0$ is
a singularity of the first kind, where
$\tilde\gamma=(\gamma_2,\gamma_3)$. In what follows, we assume
$\det(\tilde\gamma'(0),\tilde\gamma''(0))\ne0$. We have
$$
\lambda=v,\quad \eta=\partial_u-\partial_v.
$$

Let $\kappa_s^f,\kappa_\nu^f,\kappa_t^f,\kappa_c^f$ be the
invariants $\kappa_s, \kappa_\nu, \kappa_t$ and $\kappa_c$, of $f$.

Then we have
\begin{align*}
\kappa_s^f &= -\dfrac{ \sqrt{(1+(\gamma_2')^2) (\gamma_3'')^2
       -2 \gamma_3' \gamma_2' \gamma_3'' \gamma_2''
       +(1+(\gamma_3')^2) (\gamma_2'')^2}
}
{(1+(\gamma_3')^2+(\gamma_2')^2)^{3/2}}=\kappa\\
\kappa_\nu^f&=
0\\
\kappa_t^f&= \dfrac{-\gamma_2'' \gamma_3'''+\gamma_3'' \gamma_2'''}
{(\gamma_3'')^2+(\gamma_2'')^2
+(\gamma_2' \gamma_3''-\gamma_3' \gamma_2'')^2}=\tau\\
\kappa_c^f&= \dfrac{2 (1+(\gamma_3')^2+(\gamma_2')^2)^{3/2}
(-\gamma_2'' \gamma_3'''+\gamma_3'' \gamma_2''')}
{((\gamma_3'')^2+(\gamma_2'')^2 +(\gamma_2' \gamma_3''-\gamma_3'
\gamma_2'')^2)^{5/2}}.
\end{align*}
where $\kappa$ and $\tau$ are the curvature and the torsion of the
space curve $f(u,0)=\gamma(u)$. If $\tau(0)=0$, then $f(u,v)$ has a
cuspidal cross-cap singularity at the origin and $\kappa_\nu(0)=
\kappa_t(0)=\kappa_c(0)=0$. Therefore, from the point of view of its
invariants, this cuspidal cross-cap is not generic,
since there are cuspidal cross-caps with
$\kappa_\nu\ne0$ and $\kappa_t\ne 0$.

However, we can consider the contact of this cuspidal cross-cap with
planes. The tangent cone of $f$ contains $\gamma'(0)=(1,0,0)$.
Consider the contact of $f$ with its tangent cone, that is, take a
direction ${\bf v}=(0,v_2,v_3)$ and consider the height function
$$
 H_{{\bf v}}(x,y)
=(\gamma_2(u)+v\gamma_2'(u))v_2+(\gamma_3(u)+v\gamma_3'(u))v_3.
$$
Notice that $H$ can never have an $A_3$-singularity, which is the
only generic singularity of the height function which cannot be
recovered by folding a cuspidal edge (see Remark \ref{rema3} i)). This suggests that any
cuspidal cross-cap obtained from the tangent developable of a space
curve at a point of zero torsion, which is generic in the sense of
its contact with planes, can be obtained by folding a cuspidal edge.

\noindent
ROS: Departament de Matem\`{a}tiques, Universitat de Val\`{e}ncia, c/ Dr Moliner no 50, 46100, Burjassot, Val\`{e}ncia, Spain.\\
Email: raul.oset@uv.es\\

\noindent KS: Department of Mathematics,
Kobe University, Rokko 1-1, Nada,
Kobe 657-8501, Japan\\
Email: saji@math.kobe-u.ac.jp
\end{document}